\documentclass[12pt,leqno]{article}
\usepackage{amsmath, amsthm, amsfonts, amssymb, color}
\usepackage{mathrsfs}
\usepackage[active]{srcltx}
\usepackage{enumerate}
\usepackage{amsopn} 
\usepackage{bbm} 
\usepackage{mathrsfs}

\def\R{\mathbb R}    
\def\N{\mathbb N}  
   
\def\<{\langle} \def\>{\rangle}  
  
\def\d{\text{\rm{d}}}   
\def\E{\mathbb E}
\def\P{\mathbf P}
 
\def\beg{\begin} \def\beq{\begin{equation}}  
\def\F{\mathcal F}

\def\b{\mathbf b}

\def\ee{\varepsilon}
\let\oldpm\pm
\def\pm{\pi_{{\bf m}}}

\def\L{\mathcal L}
\def\({{\Big(}}
\def\){{\Big)}}

\def\bt{\begin{theorem}}
\def\et{\end{theorem}}
\def\bl{\begin{lemma}}
\def\el{\end{lemma}}
\def\br{\begin{remark}}
\def\er{\end{remark}}
\def\bx{\begin{Example}}
\def\ex{\end{Example}}
\def\bd{\begin{definition}}
\def\ed{\end{definition}}
\def\bp{\begin{proposition}}
\def\ep{\end{proposition}}
\def\bc{\begin{corollary}}
\def\ec{\end{corollary}}

\newcommand{\be}{\begin{eqnarray}}
\renewcommand{\ee}{\end{eqnarray}}
\newcommand{\ce}{\begin{eqnarray*}}
\newcommand{\de}{\end{eqnarray*}}

\newtheorem{thm}{Theorem}[section]
\newtheorem{cor}[thm]{Corollary}

\newtheorem{lem}[thm]{Lemma}
\newtheorem{prp}[thm]{Proposition}

\newtheorem{exa}[thm]{Example}

\newtheorem{defn}[thm]{Definition}
\newtheorem{rem}[thm]{Remark}

\definecolor{wco}{rgb}{0.5,0.2,0.3}

\numberwithin{equation}{section}

\newcounter{cprop}[section]

\newcommand{\mcB}{\mathcal{B}}

\newtheorem{definition}[cprop]{Definition}
\newtheorem{remark}[cprop]{Remark}
\newtheorem{lemma}[cprop]{Lemma}
\newtheorem{proposition}[cprop]{Proposition}
\newtheorem{theorem}[cprop]{Theorem}
\newtheorem{corollary}[cprop]{Corollary}

\renewcommand{\and}{\quad\textrm{ and }\quad}

\renewcommand{\P}{\mathbb{P}}
\renewcommand{\a}{\alpha}
\renewcommand{\b}{\beta}

\renewcommand{\L}{\Lambda}
\renewcommand{\l}{\lambda}

\renewcommand{\o}{\omega}
\renewcommand{\O}{\Omega}
\renewcommand{\t}{\theta}

\newcommand{\mcA}{\mathcal{A}}

\newcommand{\mcL}{\mathcal{L}}

\newcommand{\mcF}{\mathcal{F}}

\newcommand{\mcP}{\mathcal{P}}

\newcommand{\g}{\gamma}

\newcommand{\s}{\sigma}
\newcommand{\vp}{\varphi}

\newcommand{\ve}{\varepsilon}

\newcommand{\Sum}{\sum\limits}

\allowdisplaybreaks[1]

\title{{\bf Random attractors for a class of stochastic partial
 differential equations  driven by general additive noise}
\footnote{Supported in part by DFG--Internationales Graduiertenkolleg
``Stochastics and Real World Models'', the SFB-701  and  the BiBoS Research Center.
The support of  Issac Newton Institute for Mathematical
Sciences in Cambridge is also gratefully acknowledged where part of this work was done during the special semester on
``Stochastic Partial Differential Equations''.} }
\author{{\bf Benjamin Gess $^a$, Wei Liu $^a$\footnote{Corresponding author: wei.liu@uni-bielefeld.de}
, Michael R\"{o}ckner $^{a,b}$
}\\
{\footnotesize $a.$ Fakult\"at f\"ur Mathematik, Universit\"at Bielefeld,
D-33501 Bielefeld, Germany}\\
 \footnotesize{$b.$ Department of Mathematics and Statistics, Purdue University, West Lafayette, 47906 IN, USA}\\
}
\date{}
\begin{document}
\maketitle

\begin{abstract}
The existence of  random attractors for a large class of stochastic partial differential equations
 (SPDE) driven by general additive noise is established.
The main results are applied to various types of SPDE,  as e.g.  stochastic reaction-diffusion equations,
the stochastic $p$-Laplace equation and  stochastic porous media equations.
Besides classical Brownian motion, we also include space-time fractional Brownian motion
and space-time L\'evy noise as admissible random perturbations.
 Moreover, cases where the attractor consists of a single point are also investigated and bounds for the speed of attraction are obtained.
\end{abstract}
\noindent
 AMS Subject Classification:\ 35B41, 60H15, 37L30, 35B40\\
\noindent
 Keywords: Random attractor; L\'evy noise; fractional Brownian motion; stochastic evolution equations;
 porous media equations; $p$-Laplace equation; reaction-diffusion equations.

\bigbreak

\section{Introduction}

Since the foundational work in \cite{CDF97,CF94,S92} the long time behaviour of several examples of SPDE
 perturbed by additive noise has been extensively investigated by means of proving
the existence of a global random attractor (cf. e.g. \cite{BLW09,B06,CL03,CLR98,D97,F04,LG08,W09,W09b}).
However, these results address only some specific examples of SPDE of semilinear type.
To the best of our knowledge the only result concerning a non-semilinear SPDE, namely stochastic generalized porous media equations
is given in \cite{BGLR10}. In this work we provide a general result yielding the existence of a (unique) random attractor
for a large class of SPDE perturbed by general additive noise.
In particular, the result is applicable also to quasilinear equations like stochastic porous media equations
and the stochastic $p$-Laplace equation. The existence of the random attractor for the stochastic
porous medium equation (SPME) as obtained in \cite{BGLR10} is contained as a special case
(at least if the noise is regular enough, cf. Remark \ref{rmk:PME}). We also would like to point out that we include the well-studied
case of stochastic reaction-diffusion equations, even in the case of high order growth of the nonlinearity by reducing it to
the deterministic case and then applying our general results
(cf. Remark \ref{rmk:RDE} for details and comparison with previous results).
 Apart from allowing a large class of admissible drifts, we also formulate our results for general
additive perturbations, thus containing the case of Brownian motion and fractional Brownian motion (cf. \cite{GKN09,MS04}).
We emphasize, however, that the continuity of the noise in time is not necessary. Our techniques are designed so that
they also apply to c\`{a}dl\`{a}g noise. In particular, L\'{e}vy-type noises are included (cf. Section \ref{sec:applications}).
Under a further condition on the drift, we prove that the random attractor consists of a single point, i.e. the existence of a random fixed point. Hence the existence of a unique stationary solution is also obtained. 

Our results are based on the variational approach to (S)PDE. This  approach has been used
intensively in recent years to analyze SPDE driven by an infinite-dimensional Wiener process.
For general results on the existence and uniqueness of variational solutions to SPDE
 we refer to \cite{G82,KR79,LR10,P75,RRW07,Z09}.  As a typical example of an SPDE in this framework
 stochastic porous media equations have been intensively investigated in \cite{BDR08,BDR08-1,BDR09,BDR09-1,DRRW06,K06,L09,L10,RW08}.

This  paper is organized as follows. In the rest of this section we present the main results 
(Theorems \ref{thm:generation_rds}, \ref{T6.1}  and \ref{thm:single_point}) and recall  some concepts of the theory of random dynamical systems.
  The proofs of the main theorems are given in the next section.
In Section 3  we apply
the main results to various examples of SPDE such as  stochastic reaction-diffusion equations,
the stochastic $p$-Laplace equation and  stochastic porous medium equations with general additive noise.

Now let us  describe our framework  and the main results.  Let
  $$V\subseteq H\equiv H^*\subseteq V^*$$
be a Gelfand triple, $i.e.$  $(H, \<\cdot,\cdot\>_H)$ is a separable Hilbert space and is identified with
its dual space $H^*$ by the Riesz isomorphism $i: H\rightarrow H^*$, $V$ is a reflexive  Banach space such that
 it is continuously and densely embedded into $H$. $_{V^*}\<\cdot,\cdot\>_V$ denotes the dualization
between $V$ and its dual space $V^*$. Let $A: V \to V^*$ be measurable,
 $(\Omega,\mathcal{F},\mathcal{F}_t,\mathbb{P})$ be a filtered probability space and
$(N_t)_{t \in \R}$ be a $V$-valued adapted stochastic process.
 For $[s,t] \subseteq \R$ we consider the following stochastic evolution equation
\begin{align}\label{equation 6.1}
   \d X_r &= A(X_r)\d r+ dN_r, \  r \in [s,t], \\
   X_s    &= x\in H. \nonumber
\end{align}
If $A$ satisfies the standard monotonicity and coercivity conditions (cf. $(H1)-(H4)$ below) we shall prove
the existence and uniqueness of solutions to $(\ref{equation 6.1})$ in the sense of Definition \ref{def:soln}.

  Suppose that there exist  $\alpha>1$ and constants $\delta>0$, $K,C \in \R$ such that the following conditions hold for all $v,v_1,v_2\in V$:
 \begin{enumerate}
 \item [$(H1)$] (Hemicontinuity)
      The map  $ s\mapsto { }_{V^*}\<A(v_1+s v_2),v\>_V$ is  continuous on $\mathbb{R}$.
  \item [$(H2)$] (Monotonicity)
     $$2{  }_{V^*}\<A(v_1)-A(v_2), v_1-v_2\>_V
     \le C\|v_1-v_2\|_H^2. $$
  \item [$(H3)$] (Coercivity)
    $$ 2{ }_{V^*}\<A(v), v\>_V +\delta
    \|v\|_V^{\alpha} \le C + K \|v\|_H^2.$$
  \item[$(H4)$] (Growth)
   $$ \|A(v)\|_{V^*} \le  C(1 + \|v\|_V^{\alpha-1}).$$
  \end{enumerate}

We can now define the notion of a solution to \eqref{equation 6.1}.

\beg{defn}\label{def:soln} An $H$-valued, $(\F_t)$-adapted process $\{X_r\}_{r\in [s,t]}$ is called a  solution of $(\ref{equation 6.1})$ if $X_\cdot(\o) \in L^\alpha([s,t]; V) \cap L^2([s,t]; H)$ and
  $$X_r(\o)=x+\int_s^r A(X_u(\o))\d u + N_r(\o)-N_s(\o) $$
holds for all $r \in [s,t]$ and all $\o \in \O$.
\end{defn}

Since the solution to \eqref{equation 6.1} will be constructed via a transformation of \eqref{equation 6.1}
into a deterministic equation
(parametrized by $\o$) we can allow very general additive stochastic perturbations.
In particular, we do not have to assume the noise to be a martingale or a Markov process.

Since the noise is not required to be Markovian, the solutions to the SPDE cannot be expected to
define a Markov process. Therefore, the approach to study long-time behaviour of solutions to SPDE
via invariant measures and ergodicity of the associated semigroup  is not an option here. In particular, the results from \cite{KS04} cannot be applied to prove that the attractor consists of a single point. Consequently,
our analysis  is instead based on the framework of random dynamical systems (RDS), which more or less
 requires the driving process to have stationary increments (cf. Lemma \ref{lem:gen_metric_dns}).

Let $((\O,\mcF,\P),(\theta_t)_{t \in \R})$ be a metric dynamical system, i.e. $(t,\o) \mapsto \theta_t(\o)$ is $\mcB(\R) \otimes \mcF/\mcF$ measurable, $\theta_0 =$ id, $\theta_{t+s} = \theta_t \circ \theta_s$ and $\theta_t$ is $\P$-preserving, for all $s,t \in \R$.
  \begin{enumerate}
    \item [$(S1)$] (Strictly stationary increments) For all $t,s \in \R$, $\o \in \O$:
                    $$N_t(\o)-N_s(\o) = N_{t-s}(\t_s \o)-N_0(\t_s \o).$$
    \item [$(S2)$] (Regularity) For each $\o \in \O$,
                    $$N_\cdot(\o) \in L^\a_{loc}(\R;V)\cap L^2_{loc}(\R;H)$$
(with the same $\alpha>1$ as in $(H3)$).
    \item [$(S3)$] (Joint measurability) $N: \R \times \O \to V$ is $\mcB(\R) \otimes \mcF/\mcB(V)$ measurable.
  \end{enumerate}

\begin{rem}
  Although we do not explicitly assume $N_t$ to have c\`{a}dl\`{a}g paths, in the applications the underlying metric dynamical system $((\O,\mcF,\P),(\theta_t)_{t \in \R})$ is usually
defined as the space of all c\`{a}dl\`{a}g functions endowed with a topology making the Wiener shift $\t: \R\times\Omega \to \Omega;\ \theta_t(\omega)=\omega(\cdot+t)-\omega(t)$ measurable and the probability measure $\P$ is given by the distribution of the noise $N_t$. Thus, in the applications we will always require $N_t$ to
have c\`{a}dl\`{a}g paths.
\end{rem}

We now recall the notion of a random dynamical system. For more details concerning the theory of random dynamical systems
we refer to \cite{CDF97,CF94}.

\begin{defn}
  Let $(H,d)$ be a complete and separable metric space.
  \begin{enumerate}
  \item[(i)] A random dynamical system (RDS) over $\theta_t$ is a measurable map
    \begin{align*}
     \vp: \R_+ \times H \times \O \to H;\
     (t,x,\o)      \mapsto \vp(t,\o)x
    \end{align*}
    such that $\vp(0,\o) =$ id and
      \[ \vp(t+s,\o) = \vp(t,\theta_s \o) \circ \vp(s,\o), \]
    for all $t,s \in \R_+$ and $\o \in \O$. $\vp$ is said to be a continuous RDS if $x \mapsto \vp(t,\o)x$ is continuous for all $t \in \R_+$ and $\o \in \O$.
  \item[(ii)]
  A stochastic flow is a family of mappings $S(t,s;\o): H \to H$, $-\infty < s \le t < \infty$, parametrized by $\o$ such that
  \begin{align*}
    (t,s,x,\o) \mapsto S(t,s;\o)x
  \end{align*}
  is $ \mcB(\R) \otimes \mcB(\R) \otimes \mcB(H) \otimes \mcF/\mcB(H)$-measurable and
  \begin{align*}
    &S(t,r;\o)S(r,s;\o)x = S(t,s;\o)x, \\
    &S(t,s;\o)x = S(t-s,0;\t_s\o)x,
  \end{align*}
  for all $s \le r \le t$ and all $\o \in \O$. $S$ is said to be a continuous stochastic flow if $x \mapsto S(t,s;\o)x$ is continuous for all $s \le t$ and $\o \in \O$.
  \end{enumerate}
\end{defn}

In order to apply the theory of RDS and in particular to apply Proposition \ref{prop:suff_criterion} below,
we first need to define the RDS  associated with \eqref{equation 6.1}. For this we consider the unique $\o$-wise solution
(denoted by $Z(\cdot, s;\omega)x$) of
\begin{equation}\label{e1}
    Z_t=x-N_s(\o)+\int_s^t A(Z_r+N_r(\omega)) \d r,\ t \ge s,
\end{equation}
and then define
\begin{align}
    S(t,s;\o)x &:= Z(t,s;\o)x + N_t(\o), \label{eqn:S-def}\\
    \vp(t,\o)x &:= S(t,0;\o)x = Z(t,0;\o)x + N_t(\o). \label{eqn:vp-def}
\end{align}
Note that $S(\cdot,s;\omega)$ satisfies
  $$ S(t,s;\o)x = x + \int_s^t A(S(r,s;\o)x) \d r + N_t(\o) - N_s(\o),$$
for each fixed $\o \in \O$ and all $t \ge s$. Hence $S(t,s;\o)x$ solves \eqref{equation 6.1} in the sense of Definition \ref{def:soln}.

\begin{thm}\label{thm:generation_rds}
  Under the assumptions $(H1)$-$(H4)$ and $(S1)$-$(S3)$, $S(t,s;\o)$ defined in \eqref{eqn:S-def} is a continuous stochastic flow and $\vp$ defined in \eqref{eqn:vp-def} is a continuous random dynamical system.
\end{thm}

For the proof of Theorem \ref{thm:generation_rds}  as well as the other theorems in this section we refer to the next section.

With the notion of an RDS above we can now recall the stochastic generalization of notions of absorption,
attraction and $\O$-limit sets (cf. \cite{CDF97,CF94}).

\begin{defn}\label{def:rds_basics}
  \begin{enumerate}
  \item[(i)] A (closed) set-valued map $K: \O \to 2^H$ is called measurable if $\omega\mapsto K(\omega)$ takes values in
the closed subsets of $H$ and  for all $x \in H$ the map $\o \mapsto d(x,K(\o))$ is measurable, where for nonempty sets $A,B \in 2^H$ we set
$$d(A,B)=\sup\limits_{x \in A} \inf\limits_{y \in B} d(x,y);   \    \   d(x,B) = d(\{x\},B).$$
 A measurable (closed) set-valued map is also called a (closed) random set.
  \item[(ii)] Let $A$, $B$ be random sets. $A$ is said to absorb $B$ if $\P$-a.s. there exists an absorption time $t_B(\o)$ such that for all $t \ge t_B(\o)$
        \[ \vp(t,\theta_{-t}\o)B(\theta_{-t}\o) \subseteq A(\o).\]
        $A$ is said to attract $B$ if
        \[ d(\vp(t,\theta_{-t}\o)B(\theta_{-t}\o),A(\o)) \xrightarrow[t \to \infty]{} 0,\ \P\text{-a.s. }.\]
  \item[(iii)] For a random set $A$ we define the $\O$-limit set to be
        \[ \O_A(\o)=\O(A,\o) = \bigcap_{T \ge 0} \overline{ \bigcup_{t \ge T} \vp(t,\theta_{-t}\o)A(\theta_{-t}\o)}. \]
  \end{enumerate}
\end{defn}

\begin{defn}
  A random attractor for an RDS $\vp$ is a compact random set $A$ satisfying $\P$-a.s.:
  \begin{enumerate}
   \item[(i)] $A$ is invariant, i.e. $\vp(t,\o)A(\o)= A(\theta_t \o)$ for all $t > 0$.
   \item[(ii)] $A$ attracts all deterministic bounded sets $B \subseteq H$.
  \end{enumerate}
\end{defn}

Note that by \cite{C99} the random attractor for an RDS is uniquely determined.

The following proposition yields a sufficient criterion for the existence of a random attractor of an RDS.
\begin{prp}(cf. \cite[Theorem 3.11]{CF94})\label{prop:suff_criterion}
  Let $\vp$ be an RDS and assume the existence of a compact random set $K$ absorbing every deterministic bounded set $B \subseteq H$.
Then there exists a random attractor $A$, given by
  \[ A(\o) = \overline{ \bigcup_{B \subseteq H,\ B \text{ bounded}} \O_B(\o)}.\]
\end{prp}
\begin{rem}
 In fact, it is known that  the existence of a random attractor is equivalent to  the existence of a compact  attracting  random set
(see \cite{CDS} for more equivalent conditions).
\end{rem}

We aim to apply Proposition \ref{prop:suff_criterion} to prove the existence
of a random attractor for the RDS associated with (\ref{equation 6.1}).
Thus, we need to prove the existence of a compact globally absorbing random set $K$. To show the existence of such a set for (\ref{equation 6.1}), we require some additional assumptions to derive an a priori estimate of the solution in a norm $\|\cdot\|_S$, which is stronger than the norm $\|\cdot\|_H$.

\begin{enumerate}
  \item [$(H5)$] Suppose  there is a subspace $(S, \|\cdot\|_S)$ of $H$ such that the embedding $V\subseteq S$ is continuous and $S\subseteq H$ is compact. Let $T_n$ be positive definite self-adjoint operators on $H$ such that
  $$\<x , y\>_n:=\<x , T_n y\>_H,\ x,y\in H, n\ge 1,$$
  define a sequence of new inner products on $H$. Suppose that the induced norms $\|\cdot\|_n$ are all equivalent to $\|\cdot\|_H$ and for all $x\in S$ we have
    $$\|x\|_n\uparrow\|x\|_S\ \text{as}\  n\rightarrow\infty.$$
  Moreover, we assume that $T_n:V\rightarrow V,\ n\ge 1,$ are continuous and that there exists a constant $C>0$ such that
    \begin{equation}\label{condition 1}
      2{ }_{V^*}\langle A(v),  T_nv\rangle_V\le C(\|v\|_n^2+1), \ v\in V,
    \end{equation}
    and
    \begin{equation}\label{noise condition}
       \sup_{n \in \N} \int_{-1}^0 \|T_n N_t\|_V^\a dt \le C.
    \end{equation}
  \end{enumerate}

\begin{rem} (1)
Assumption $(H5)$ looks quite abstract at first glance.  But it is applicable to a large class of SPDE within the
variational framework,  as e.g. stochastic reaction-diffusion equations, stochastic porous media equations and the
 stochastic $p$-Laplace equation (see  Section 3 for  more examples).

(2) Under  assumption (\ref{condition 1}) the following regularity property of  solutions to general SPDE driven by a Wiener process was established in \cite{L10-2}:
  $$ \mathbb{E}\sup_{s\in[0,t]}\|X_s\|_S^2< \infty,\ \text{for all}\ t>0. $$
\end{rem}

In order to prove the existence of a random attractor, we  need to assume some growth condition on the paths of the noise.

  \begin{enumerate}
    \item [$(S4)$] (Subexponential growth) For $\P$-a.a. $\o \in \O$ and $|t| \to \infty$, $N_t(\o)$ is of subexponential growth, i.e. $\|N_t(\o)\|_V = o(e^{\lambda |t|})$ for every $\lambda >0$.
  \end{enumerate}

\begin{thm}\label{T6.1}
  Suppose $(H1)$-$(H5)$ hold for $\alpha=2,K=0$ or for $\alpha>2$, and that $(S1)$-$(S4)$ are satisfied. Then the RDS $\varphi$ associated with SPDE $(\ref{equation 6.1})$ has a compact random attractor.
\end{thm}

\begin{rem}
 $(H1)$-$(H4)$ are the classical monotonicity and coercivity conditions
for the existence and uniqueness of solutions to (\ref{equation 6.1}). It can be replaced by some much weaker
assumptions (e.g. local monotonicity) according to some recent results in \cite{LR10,L11}.
 The existence of random attractors for SPDE with locally monotone coefficients (cf. \cite{LR10,L11})
will be the subject for
 future investigation.
 \end{rem}

In order to make the proof easier to follow, we first give a quick outline. By Proposition \ref{prop:suff_criterion} we only need to prove the existence of a compact globally absorbing random set $K$. This set will be chosen as
$$K(\o) := \overline{ B_S (0, r (\o ) )}^H, $$
where $B_S (0, r)$ denotes the ball with center $0$ and radius $r$ (depending on $\o$) in $S$. Since $S \subseteq H$ is a compact embedding, $K$ is a compact random set in $H$. Note that
   $$\vp(t,\theta_{-t}\o) = S(t,0;\theta_{-t}\o) = S(0,-t;\o).$$
Hence we need  pathwise bounds on $S_0(=S(0,-t;\o))$ in the $S$-norm. In order to get such estimates we consider the norms $\|\cdot\|_n$ on $H$ for which we can apply It\^o's formula.

Under the following stronger monotonicity condition we prove that the random attractor consists of a single point:
  \begin{enumerate}
  \item [$(H2^\prime)$] There exist constants $\beta \ge 2$ and $\l > 0$ such that
    $$2{  }_{V^*}\<A(v_1)-A(v_2), v_1-v_2\>_V \le - \l \|v_1-v_2\|_H^{\beta},\ \forall v_1,v_2 \in V. $$
\end{enumerate}

\begin{thm}\label{thm:single_point}
  Suppose that $(H1)$,$(H2')$,$(H3)$,$(H4)$ and $(S1)$-$(S3)$ hold. If $\b = 2$ also suppose $(S4)$ holds.
 Then the RDS $\varphi$ associated with  SPDE $(\ref{equation 6.1})$ has a compact random attractor $\mcA(\o)$ consisting of a single point:
    \[  \mcA(\o) = \{\eta_0(\o)\}. \]
  In particular, there is a unique random fixed point $\eta_0(\o)$ and a unique invariant random measure $\mu_\cdot \in \mcP_\O(H)$ which is given by
    \[ \mu_\o = \delta_{\eta_0(\o)}, \hskip10pt \P \text{-a.s. .}\]
  Moreover,
  \begin{enumerate}
  \item[(i)] if $\b > 2$, then the speed of convergence is polynomial, more precisely,
           $$ \|S(t,s;\o)x-\eta_0(\theta_t\o)\|_H^2 \le \left\{\frac{\l}{2} (\b-2) (t-s)\right\}^{-\frac{2}{\b-2}},
\ \forall x\in H.$$

  \item[(ii)]  if $\b = 2$, then the speed of convergence is exponential. More precisely, for every $\eta \in (0,\l)$ there is a random variable $K_\eta$ such that
           $$ \|S(t,s;\o)x-\eta_0(\theta_t\o)\|_H^2 \le 2 \left( K_\eta(\o) + \|x\|_H^2 \right) e^{(\l - \eta) s} e^{-\l t}, \
\forall x\in H.$$
  \end{enumerate}
\end{thm}

\begin{rem}
  (1) In case $\b > 2$ we recover the optimal rate of convergence found in the deterministic case in \cite{AP81}
for the porous media equation.

  (2) Note that $(H5)$ and for $\b > 2$ the growth condition for the noise $(S4)$ are not required in Theorem \ref{thm:single_point}.
\end{rem}


\section{Proofs of the main theorems}
\subsection{ Proof of Theorem \ref{thm:generation_rds} }
We need to show that the solution to (\ref{equation 6.1}) generates a random dynamical system.
In order to verify the cocycle property, we use the standard transformation to rewrite the SPDE (\ref{equation 6.1})
as a PDE with a random parameter. This is the reason why we need to restrict $N_t$ to take values in $V$ instead of $H$.
 For simplicity, in the proof   the generic constant $C$ may change from line to line.

\begin{proof}
Consider the PDE \eqref{e1} with random parameter $\omega \in \O$ and let
  $$\tilde{A}_\o(t,v) := A(v+N_t(\o)),$$
which is a well-defined operator from $V$ to $V^*$ since $N_t(\o) \in V$.
To obtain the  existence and uniqueness of  solutions to \eqref{e1} we
check the assumptions of \cite[Theorem 4.2.4]{PR07}. Since $N_\cdot(\o)$ is measurable,
$\tilde{A}_\o(t,v)$ is $\mcB(\R)\otimes \mcB(V)$ measurable. It is obvious that hemicontinuity
and (weak) monotonicity hold for $\tilde{A}_\o$. For the coercivity,  using $(H3)$, $(H4)$ and Young's inequality we have
\begin{equation}\begin{split}\label{coercivity}
&2{ }_{V^*}\<\tilde{A}_\o(t,v), v\>_V=2{ }_{V^*}\<A(v+N_t(\o)), v+N_t(\o)-N_t(\o)\>_V\\
\le & -\delta \|v+N_t(\o) \|_V^\alpha +K \|v+N_t(\o)  \|_H^2+C- 2{ }_{V^*}\<A(v+N_t(\o)), N_t(\o)\>_V\\
\le & -\delta \|v+N_t(\o) \|_V^\alpha +K \|v+N_t(\o)  \|_H^2+C+C \left(1+\|v+N_t(\o)\|_V^{\alpha-1}\right)\|N_t(\o)\|_V\\
\le & -\frac{\delta}{2} \|v+N_t(\o)  \|_V^\alpha +K \|v+N_t(\o)  \|_H^2+C\left(1+\|N_t(\o)\|_V^\alpha\right)\\
\le & -2^{-\alpha}\delta\|v\|_V^\alpha+2K\|v\|_H^2+f_t,
\end{split}
\end{equation}
where $f_t=2K\|N_t(\o)\|_H^2+C+C\|N_t(\o)\|_V^\alpha \in L^1_{loc}(\R)$ by $(S2)$.

 The growth condition also holds for $\tilde{A}_\o$ since
\begin{align*}
  \|\tilde{A}_\o(t,v)\|_{V^*}
  &= \|A(v+N_t(\o))\|_{V^*} \\
&\le C(1+\|v+N_t(\o) \|_V^{\alpha-1}) \\
  &\le f_t^{(\alpha-1)/\alpha} + C\|v\|_V^{\alpha-1}.
\end{align*}
Therefore, according to the classical results in \cite{KR79,PR07} (applied to the deterministic case),
(\ref{e1}) has   a unique solution
$$Z(\cdot,s;\o)x \in L^{\a}_{loc}([s,\infty);V) \cap C([s,\infty),H)$$
  and $x \mapsto Z(t,s;\o)x$ is continuous in $H$ for all $s\le t$ and  $\o \in \O$.

Now we define $S(t,s;\o)x$ by \eqref{eqn:S-def} and $\vp(t,\o)x$ by \eqref{eqn:vp-def}.
For fixed $s,\o,x$ we abbreviate $S(t,s;\o)x$ by $S_t$ and $Z(t,s;\o)x$ by $Z_t$. By the pathwise uniqueness of the solution to equation \eqref{e1} and $(S1)$ we have
\begin{align}
  S(t,s;\o) &= S(t,r;\o)S(r,s;\o), \nonumber\\
  S(t,s;\o) &= S(t-s,0;\theta_s\o), \label{eqn:cocycle_prop}
\end{align}
for all $r,s,t \in \R$ and all $\o \in \O$.

It remains to prove the measurability of $\vp: \R \times H \times \O \to H$.
By \eqref{eqn:cocycle_prop} this also implies the measurability of $(t,s,x,\o) \mapsto S(t,s;\o)x$.
Since $\vp(t,\o)x = Z(t,0;\o)x + N_t(\o)$ and by  $(S3)$ it is
sufficient to show the measurability of $(t,x,\o) \mapsto Z(t,0;\o)x$.
Note that the maps $t \mapsto Z(t,0;\o)x$ and $x \mapsto Z(t,0;\o)x$ are continuous,
 thus we only need to prove the measurability of $\o \mapsto Z(t,0;\o)x$.

Let $x \in H$ and $t \in \R$ be arbitrary, fix and choose some interval $[s_0,t_0] \subseteq \R$ such that $t \in (s_0,t_0)$. By the proof of the existence and uniqueness of solutions to (\ref{e1}) we know that $Z(t,0;\o)x$ is the weak limit of a subsequence of the Galerkin approximations $Z^n(t,0;\o)x$ in $L^\a([s_0,t_0];V)$. Since every subsequence of $Z^n(t,0;\o)x$ has a subsequence weakly converging to $Z(t,0;\o)x$, this implies that the whole sequence of
Galerkin approximants $Z^n(t,0;\o)x$ weakly converges to $Z(t,0;\o)x$ in $L^\a([s_0,t_0];V)$.

Let $\vp_k \in C^\infty_0(\R)$ be a Dirac sequence  with $supp(\vp_k) \subseteq B_{\frac{1}{k}}(0)$. Then $(\vp_k \ast Z^n(\cdot,0;\o)x)(t)$ is well defined for $k$ large enough.
For each such $k \in \N$ and $h \in H$ we have
$$(\vp_k \ast \< Z^n(\cdot,0;\o)x,h \>_H )(t) \to (\vp_k \ast \< Z(\cdot,0;\o)x, h\>_H)(t), \ n\rightarrow \infty.$$
 Since $\o \mapsto Z^n(\cdot,0;\o)x \in L^\a([s_0,t_0];V)$ is measurable, so is $\o \mapsto (\vp_k \ast Z^n(\cdot,0;\o)x)(t)$. Consequently, $\o \mapsto (\vp_k \ast \< Z(\cdot,0;\o)x,h\>_H)(t)$ is measurable as it is the $\o$-wise limit of $(\vp_k \ast \< Z^n(\cdot,0;\o)x,h\>_H)(t)$. We know that $r \mapsto Z(r,0;\o)x$ is continuous in $H$. Therefore, $(\vp_k \ast \< Z(\cdot,0;\o)x,h\>_H)(t) \to \< Z(t,0;\o)x,h\>_H$ and  the measurability of $\o \mapsto (\vp_k \ast \< Z(\cdot,0;\o)x,h\>_H)(t)$ implies the
 measurability of $\o \mapsto \< Z(t,0;\o)x,h\>_H$.

Since this is true for all $h \in H$ and $\mcB(H)$ is generated by $\s(\{\< h,\cdot \>_H|\ h \in H\})$, this implies the measurability of $\o \mapsto Z(t,0;\o)x$.
This finishes the proof that $\vp$ defines a continuous RDS and consequently, that $S$ defines a continuous stochastic flow.

Note that adaptedness of $S_t$ to $\mcF_t$ can be shown in the same way as the measurability of $\vp$.
\end{proof}

\subsection{Proof of Theorem \ref{T6.1}}
Since in Theorem \ref{thm:generation_rds} we have proved that $\vp$ defines an RDS, we can apply Proposition \ref{prop:suff_criterion} to show the existence of a random attractor for $\vp$. For this we follow the procedure outlined in the introduction. First we prove the absorption of $Z(t,s;\o)x$ in $H$ at time $t=-1$.

\begin{lem}\label{lem2}
  Suppose $(H1)$-$(H4)$ hold for  $\alpha=2,K=0$ or for $\alpha>2$ and that $(S1)$-$(S4)$ are satisfied. Then there exists a random radius $r_1(\omega)>0$  such that for all $\rho>0$, there exists $\bar{s}\le -1$
in such a way that  $\mathbb{P}-$a.s. we have
    $$ \|Z(-1,s;\o)x \|_H^2\le r_1^2(\omega),$$
 which holds for all $s \le \bar{s}$ and all $x \in H$ with $\|x \|_H\le\rho$ .
\end{lem}
\begin{proof}
  By the coercivity of $\tilde{A}_\o$ proved in the previous section (see (\ref{coercivity})) we have
  \begin{equation}\begin{split}
                  \frac{\d}{\d t}\|Z_t\|_H^2=2{ }_{V^*}\<\tilde{A}_\o(t,Z_t), Z_t\>_V
  \le -\delta_0\|Z_t\|_V^\alpha+2K\|Z_t\|_H^2+f_t,
                  \end{split}
  \end{equation}
  where $\delta_0=2^{-\alpha}\delta>0$ and $f_t=2K\|N_t(\o)\|_H^2+C( \|N_t(\o)\|_V^\alpha+1)$.

  If $\alpha>2$ or $\alpha=2, K=0$, then there exist constants $\lambda>0$ and $C$ such that
  \begin{equation}\label{ine1}
  \frac{\d}{\d t}\|Z_t\|_H^2+\frac{\delta_0}{2}\|Z_t\|_V^\alpha\le -\lambda\|Z_t\|_H^2+f_t+C .
  \end{equation}
 By  Gronwall's Lemma  for all $s\le -1$ we have,
  \begin{equation}
  \begin{split}
    \|Z_{-1}\|_H^2 &\le e^{-\lambda(-1-s)}\|Z_{s}\|_H^2+\int_{s}^{-1} e^{-\lambda(-1-r)}(f_r+C)\d r\\
  &\le 2e^{-\lambda(-1-s)}\|x\|_H^2+2e^{-\lambda(-1-s)}\|N_s(\o)\|_H^2
  +\int_{-\infty}^{-1} e^{-\lambda(-1-r)}(f_r+C)\d r.
  \end{split}
  \end{equation}

  By $(S4)$, i.e. the subexponential growth of $N_t(\o)$ for $t \to -\infty$ we know that
 the following quantity is finite for all $\omega\in \Omega$,
    $$r_1^2(\omega)=2+2\sup_{r\le -1}e^{-\lambda(-1-r)}\|N_r(\omega)\|_H^2+\int_{-\infty}^{-1} e^{-\lambda(-1-r)}(f_r(\omega)+C)\d r. $$
  Applying $(S3)$, i.e. the joint measurability of $N$ in $(t,\o)$, $r_1(\o)$ is measurable and then the assertion follows by taking some $\bar{s}<-1$ such that $e^{-\lambda(-1-\bar{s})}\rho^2\le 1$.
\end{proof}

\begin{rem}\label{rmk:intermediate_bound}
(\ref{ine1}) also implies the following estimate for the $V$-norm
\begin{equation}\label{ine2}
\frac{\delta_0}{2}\int_{-1}^0\|Z_r\|_V^\alpha \d r\le \|Z_{-1}\|_H^2+\int_{-1}^0 (f_r+C)\d r.
\end{equation}
\end{rem}
The next step is to show compact absorption of $Z(t,s;\omega)$ at time $t=0$.
We proceed by using the approximation scheme  indicated in the outline of proof.
By defining $H_n:=(H,\<\cdot,\cdot\>_n)$ (see $(H5)$) we obtain a sequence of new Gelfand triples
  $$V\subseteq H_n\equiv H_n^* \subseteq V^*.$$
Note that we use different Riesz maps $i_n: H_n\rightarrow H_n^*$ to identify $H_n \equiv H_n^*$ in these Gelfand triples. Let $i$ 
denote the Riesz map for $H\equiv H^*$. Now we recall the following  lemma, which is proved in \cite{L10-2}.

\beg{lem}\label{lem4}
  If $T_n: V\rightarrow V$ is continuous, then $i_n\circ i^{-1}: H^*\rightarrow H_n^*$ is continuous w.r.t.
 $\|\cdot\|_{V^*}$. Therefore, there exists a unique continuous extension $I_n$ of  $i_n\circ i^{-1}$ to all of
 $V^*$ such that
  \begin{equation}\label{equality}
    \ _{V^*}\langle I_nf,  v\rangle_V={ }_{V^*}\langle f,  T_nv\rangle_V,\ \ f\in V^*,\ v\in V.
  \end{equation}
\end{lem}

\begin{lem}\label{lem:cmp_absorption} Suppose the assumptions of Theorem \ref{T6.1} hold. Then there exists a random radius $r_2(\omega)>0$ such that for all $\rho>0$, there exists $\bar{s}\le -1$ in such a way that $\mathbb{P}-a.s.$ we have
  $$ \|Z(0,s;\o)x\|_S^2 \le r_2^2(\omega), $$
which holds for all $s \le \bar{s}$ and all $x\in H$ with $\|x\|_H\le\rho$.
\end{lem}
\begin{proof}
  Using the  operator $I_n: V^*\rightarrow H_n^*$ we consider the following equation
  $$  \frac{\d}{\d t}Z_t=I_n A(Z_t+N_t),  $$
  which is well defined on the new Gelfand triple
  $$ V\subseteq H_n\equiv H_n^*\subseteq V^*.   $$
  By Lemma \ref{lem4}, (\ref{condition 1}) and $(H4)$ we have
  \begin{align*}
    \frac{\d}{\d t}\|Z_t\|_n^2
    &=2{ }_{V^*}\<I_n A(Z_t+N_t),Z_t\>_V \\
    &=2{ }_{V^*}\<A(Z_t+N_t), T_nZ_t\>_V\\
    &\le C(\|Z_t+N_t\|_n^2+1)-2{ }_{V^*}\<A(Z_t+N_t),T_nN_t\>_V\\
    &\le C(\|Z_t+N_t\|_n^2+1)+2\left( \frac{\alpha-1}{\alpha}\|A(Z_t+N_t)\|_{V^*}^{\frac{\alpha}{\alpha-1}}
      +\frac{1}{\alpha}\|T_nN_t\|_V^\alpha  \right)\\
    &\le C \left(\|Z_t\|_n^2+\|Z_t\|_V^\alpha \right)
      +C \left(1+\|N_t\|_n^2+\|N_t\|_V^\alpha+\|T_nN_t\|_V^\alpha  \right)\\
    &\le C \left(\|Z_t\|_n^2+\|Z_t\|_V^\alpha \right)
      + g_t^{(n)},
  \end{align*}
  where $C$ is some positive constant and
$$g_t^{(n)} := C \left(1+\|N_t\|_S^2+\|N_t\|_V^\alpha+\|T_n N_t\|_V^\alpha  \right).$$
 Then  Gronwall's Lemma implies that  for all $s\le 0$,
  \begin{align*}
    \|Z_0\|_n^2
    &\le e^{-C s}\|Z_s\|_n^2+ C \int_s^0 e^{-C r}\|Z_r\|_V^\alpha\d r  + \int_s^0 e^{-C r}g_r^{(n)} \d r.
  \end{align*}
   Integrating on $s$ over $[-1,0]$ and using \eqref{noise condition} we have
  \begin{align*}
    \|Z_0\|_n^2
    &\le \int_{-1}^0\left(e^{-C r}\|Z_r\|_S^2+ C e^{-C r}\|Z_r\|_V^\alpha\right)\d r + \int_{-1}^0  e^{-C r}g_r^{(n)} \d r \\
    &\le \int_{-1}^0\left(e^{-C r}\|Z_r\|_S^2+ C e^{-C r}\|Z_r\|_V^\alpha\right)\d r + C_1,
  \end{align*}
where $C_1$ is a finite constant.

  Note that $\alpha\ge 2$ and $\|\cdot\|_S \le C\|\cdot\|_V $, hence by taking $n\rightarrow\infty$ and using \eqref{ine2} we have
  \begin{align*}
    \|Z_0\|_S^2
    &\le C \int_{-1}^0 e^{-C r} \left(1+\|Z_r\|_V^\alpha\right)\d r +  C_1 \\
    &\le C_2 \|Z_{-1}\|_H^2 + C_2,
  \end{align*}
   where $C_2 >0$ is a constant.
   Now the assertion follows from Lemma \ref{lem2}. 
\end{proof}

\textbf{Proof of Theorem \ref{T6.1}:}
By Lemma \ref{lem:cmp_absorption} there exists  $r_2(\o) > 0$ such that for all $\rho > 0$ there exists $\bar{s}\le -1$ in such a way that $\mathbb{P}-a.s.$
\begin{align*}
  \| S(0,s;\o)x \|_S
  &= \| Z(0,s;\o)x + N_0(\o)\|_S \\
  &\le \| Z(0,s;\o)x \|_S + \| N_0(\o)\|_S \\
& \le r_2(\o) + \| N_0(\o)\|_S
\end{align*}
holds for all $s \le \bar{s}$ and all $x\in H$ with $\|x\|_H\le\rho$.

Hence $S(t,s;\o)x$ is absorbed at time $t=0$ by the compact random set
$$K(\o)=\overline{B_S(0,r_2(\o) + \| N_0(\o)\|_S)} .$$
 By Proposition \ref{prop:suff_criterion} this implies the existence of a random attractor
for the RDS $\vp$ associated with (\ref{equation 6.1}).
\qed

\subsection{Proof of Theorem $\ref{thm:single_point}$}

The proof of the  first  lemma is mainly based on \cite[Theorem 5.1]{BGLR10}. The strong monotonicity condition $(H2')$ leads to the following strong contraction property.

\begin{lem}\label{lemma:contraction}
  Under the assumptions of Theorem \ref{thm:single_point} with $\b > 2$, for $s_1 \le s_2 < t$, $\o \in \O$ and $x,y \in H$ we have
    \begin{align*}
      \|S(t,s_1;\o)x-S(t,s_2;\o)y\|_H^2
      &\le \left\{ \|S(s_2,s_1;\o)x- y\|_H^{2-\b} + \frac{\l}{2} (\b-2) (t-s_2)\right\}^{-\frac{2}{\b-2}} \\
      &\le \left\{\frac{\l}{2} (\b-2) (t-s_2)\right\}^{-\frac{2}{\b-2}}.
    \end{align*}
 In particular, for each $t \in \R$ there exists $\eta_t$ (independent of $x$) such that
$$\lim \limits_{s \rightarrow -\infty}S(t,s;\o)x = \eta_t(\o),$$
where the convergence holds uniformly in $x$ and $\omega$.
\end{lem}
\begin{proof}  Let $\o \in \O$, $x,y \in H$ and $s_1 \le s_2 \le s < t$, then
  \begin{align*}
    &S(t,s_1;\o)x-S(t,s_2;\o)y \\
     = & S(s,s_1;\o)x - S(s,s_2;\o)y + \int_{s}^t \left( A (S(r,s_1;\o)x)-A (S(r,s_2;\o)y) \right) \d r.
  \end{align*}
  Note that $t \mapsto S(t,s_1;\o)x-S(t,s_2;\o)y$ is continuous in $H$. By It\^{o}'s formula and $(H2^\prime)$
  \begin{align}\label{eqn:diff2}
        &\|S(t,s_1;\o)x- S(t,s_2;\o)y\|_H^2 \nonumber\\
    = &\|S(s,s_1;\o)x - S(s,s_2;\o)y\|_H^2 \nonumber\\
      & + 2 \int_{s}^t { }_{V^*}\<A(S(r,s_1;\o)x)-A(S(r,s_2;\o)y),S(r,s_1;\o)x-S(r,s_2;\o)y\>_V \d r \\
     \le & \|S(s,s_1;\o)x - S(s,s_2;\o)y\|_H^2 - \l \int_{s}^t  \|S(r,s_1;\o)x-S(r,s_2;\o)y\|_H^\b \d r. \nonumber
  \end{align}
  The idea of the rest of the proof is to compare $\|S(t,s_1;\o)x-S(t,s_2;\o)y\|_H^2$ with the solution to the ordinary differential equation
  \begin{align}\label{def:ode}
    h'(t) = -\l h(t)^{\frac{\b}{2}},\   t \ge s_2; \
    h(s_2) = \|S(s_2,s_1;\o)x- y\|_H^2.
  \end{align}
  However, since $\|S(t,s_1;\o)x-S(t,s_2;\o)y\|_H^2$ is not necessarily differentiable in $t$ we cannot apply classical comparison results.

  Let  
   \[ h_{\epsilon}(t) = \left\{\left(\|S(s_2,s_1;\o)x- y\|_H + \epsilon\right)^{2-\b} +
\frac{\l}{2} (\b-2) (t-s_2) \right\}^{-\frac{2}{\b-2}}.\]
  It is easy to show that $h_{\epsilon}$ is a solution of \eqref{def:ode} with $h_\epsilon(s_2)=(\|S(s_2,s_1;\o)x- y\|_H + \epsilon)^2$.
Now we prove that
\begin{equation}\label{est 1}
\|S(t,s_1;\o)x-S(t,s_2;\o)y\|_H^2 \le h_{\epsilon}(t),\ \ t\ge s_2.
\end{equation}
Let
\begin{align*}
  \varPhi_\epsilon(t) &= h_{\epsilon}(t) - \|S(t,s_1;\o)x-S(t,s_2;\o)y\|_H^2,\\
  \tau_\epsilon &= \inf\left\{t \ge s_2 |~ \varPhi_\epsilon(t)\le 0\right\}.
\end{align*}
 Because $\varPhi_\epsilon(s_2)>0$ and by the continuity of $\varPhi_\epsilon$ we know that $\tau_\epsilon > s_2$.
Furthermore, note that by definition we have
\begin{align*}
h_{\epsilon}(t) & \ge \|S(t,s_1;\o)x-S(t,s_2;\o)y\|_H^2, \ t\in[s_2,\tau_\epsilon]; \\
   h_{\epsilon}(t) & \le (\|S(s_2,s_1;\o)x- y\|_H + \epsilon)^2 =: c_\epsilon, \ t\ge s_2.
   \end{align*}
  If $\tau_\epsilon < \infty$, then $\varPhi_\epsilon(\tau_\epsilon) \le 0$ by the continuity of $\varPhi_\epsilon$.
Therefore, by the mean value theorem and \eqref{eqn:diff2}  for all $s_2 \le s \le t \le \tau_\epsilon$ we have,
  \begin{align*}
  \varPhi_\epsilon(t)
    &= h_{\epsilon}(t) - \|S(t,s_1;\o)x-S(t,s_2;\o)y\|_H^2 \nonumber\\
    &\ge \varPhi_\epsilon(s) - \l \int_{s}^t \left(h_\epsilon(r)^{\frac{\b}{2}} - \left( \|S(r,s_1;\o)x-S(r,s_2;\o)y\|_H^2\right)^{\frac{\b}{2}} \right) \d r \\
    &\ge \varPhi_\epsilon(s)  - \frac{\l  \b c_\epsilon^{\frac{\b-2}{2}}}{2}  \int_{s}^t \varPhi_\epsilon(r) \d r\nonumber.
  \end{align*}
  Using Gronwall's Lemma we obtain
  \begin{equation*}
  \varPhi_\epsilon(\tau_\epsilon) \ge \varPhi_\epsilon(s_2)
\exp\left[-  \frac{\l\b}{2}  c_\epsilon^{\frac{\b-2}{2}} (\tau_\epsilon-s_2)\right] > 0.
  \end{equation*}
  This contradiction implies that $\tau_\epsilon = \infty$, i.e. (\ref{est 1}) holds.

 Since (\ref{est 1}) holds for any $\epsilon >0$ we can conclude that
  \begin{align*}
    \|S(t,s_1;\o)x-S(t,s_2;\o)y\|_H^2
      &\le \left\{\|S(s_2,s_1;\o)x- y\|_H^{2-\b} + \frac{\l}{2} (\b-2) (t-s_2)\right\}^{-\frac{2}{\b-2}} \nonumber \\
    &\le \|S(s_2,s_1;\o)x- y\|_H^2 \wedge \left\{\frac{\l}{2} (\b-2) (t-s_2)\right\}^{-\frac{2}{\b-2}} \nonumber \\
    &\le \left\{\frac{\l}{2} (\b-2) (t-s_2)\right\}^{-\frac{2}{\b-2}}
  \end{align*}
 holds for any $t > s_2$.
\end{proof}

\begin{lem}\label{lemma:contraction2}
  Suppose the assumptions of Theorem \ref{thm:single_point} with $\b = 2$ and $(S4)$ hold.
Then for each $\eta \in (0,\l)$ there is an $\R_+$-valued random variable $K_\eta $ such that
  \begin{align*}
    \|S(t,s_1;\o)x- S(t,s_2;\o)y\|_H^2 \le 2 \left(\|x\|_H^2 e^{\frac{\eta}{2} s_1} + K_\eta(\o) + \|y\|_H^2 \right) e^{(\l - \eta) s_2} e^{-\l t}
  \end{align*}
  for all $s_1 \le s_2 < t$, $\o \in \O$ and $x,y \in H$. In particular, for each $t \in \R$ there exists $\eta_t$ (independent of $x$) such that
    $$\lim \limits_{s \rightarrow -\infty}S(t,s;\o)x = \eta_t(\o),$$
 where the convergence holds  uniformly in $x$ on any ball $B_H(0,r) = \{h \in H|\ \|h\|_H \le r\}$.
\end{lem}
\begin{proof}
  As in Lemma \ref{lemma:contraction} for $\o \in \O$, $x,y \in H$ and $s_1 \le s_2 \le s < t$ we obtain
  \begin{align*}
    & \|S(t,s_1;\o)x- S(t,s_2;\o)y\|_H^2 \\
    \le & \|S(s,s_1;\o)x - S(s,s_2;\o)y\|_H^2 - \l \int_{s}^t  \|S(r,s_1;\o)x-S(r,s_2;\o)y\|_H^2 \d r.
  \end{align*}
  Thus, by  Gronwall's Lemma
  \begin{align*}
    \|S(t,s_1;\o)x- S(t,s_2;\o)y\|_H^2
    &\le \|S(s_2,s_1;\o)x - y\|_H^2 e^{-\l (t-s_2)} \\
    &\le 2\left(\|S(s_2,s_1;\o)x\|_H^2 + \|y\|_H^2\right) e^{-\l (t-s_2)}.
  \end{align*}
  By \cite[Lemma 4.3.8]{PR07} $(H3)$, $(H4)$ and $(H2^\prime)$ imply that for each $\eta \in (0,\l)$ there exists a $C_\eta > 0$ such that for all $v \in V$
  \begin{equation}\label{eqn:coerc_by_mon}
    2\ _{V^*}\<A(v),v\>_V \le -\eta\|v\|^2_H + C_\eta.
  \end{equation}
  Let $\eta \in (0,\l)$ and $\tilde{\eta} = \frac{\eta + \l}{2} \in (\eta,\l)$. We use \eqref{eqn:coerc_by_mon} with $\tilde{\eta}$, $(H3)$, $(H4)$ and Young's inequality to obtain
  \begin{align*}
    &2 \ _{V^*}\< A(v+N_r), v \>_V \\
    =& 2 \ _{V^*}\< A(v+N_r), v+N_r-N_r \>_V \\
    \le & 2 \ve_1 \ _{V^*}\< A(v+N_r), v+N_r \>_V + 2 (1-\ve_1) \ _{V^*}\< A(v+N_r), v+N_r \>_V \\
      & + 2 \|A(v+N_r)\|_{V^*} \|N_r\|_V\\
   \le& \ve_1 K \|v+N_r\|_H^2 - \delta \ve_1 \|v+N_r\|_V^\a + \ve_1 C - \tilde{\eta} (1-\ve_1)\|v+N_r\|^2_H + (1-\ve_1)C_{\tilde{\eta}}  \\
      & + \ve_2 \|A(v+N_r)\|_{V^*}^{\frac{\a}{\a-1}} + C_{\ve_2} \|N_r\|_V^\a \\
    \le& (\ve_1 K - \tilde{\eta} (1-\ve_1)) \|v+N_r\|_H^2 + (\ve_2 C - \delta \ve_1) \|v+N_r\|_V^\a + \ve_1 C + (1-\ve_1)C_{\tilde{\eta}} \\
      & + \ve_2 C + C_{\ve_2} \|N_r\|_V^\a,
  \end{align*}
  where $\ve_1 \in [0,1]$, $\ve_2 > 0$ and  $C, C_{\ve_1}, C_{\ve_2} > 0$ are some constants.

Now by taking $0 < \ve_1 \le \frac{\tilde{\eta}-\eta}{\tilde{\eta} + K} \wedge 1$ and $ \ve_2 = \frac{\delta \ve_1}{C}$ we have
  \begin{align*}
    2 \ _{V^*}\< A(v+N_r), v \>_V
    &\le - \eta \|v+N_r\|_H^2 + \ve_1 C + (1-\ve_1)C_{\tilde{\eta}} + \ve_2 C + C_{\ve_2} \|N_r\|_V^\a \\
    &\le - \frac{\eta}{2} \|v\|_H^2 + \eta \|N_r\|_H^2 + \ve_1 C + (1-\ve_1)C_{\tilde{\eta}} + \ve_2 C + C_{\ve_2} \|N_r\|_V^\a \\
    &\le - \frac{\eta}{2} \|v\|_H^2 + C_\eta(r),
  \end{align*}
 where $C_\eta(r) = \eta \|N_r\|_H^2 + C_{\ve_2} \|N_r\|_V^\a+ \ve_1 C + (1-\ve_1)C_{\tilde{\eta}} + \ve_2 C$.

 Hence for all  $t_2 \ge t_1 \ge s$,
  \begin{align*}
    \|Z(t_2,s;\o)x\|_H^2
    &= \|Z(t_1,s;\o)x\|_H^2 + 2\int_{t_1}^{t_2} { }_{V^*}\<A(Z(r,s;\o)x + N_r(\o)),Z(r,s;\o)x\>_V \d r \\
    &\le \|Z(t_1,s;\o)x\|_H^2 - \frac{\eta}{2} \int_{t_1}^{t_2} \|Z(r,s;\o)x\|_H^2\d r
+ \int_{t_1}^{t_2} C_\eta(r) \d r.
  \end{align*}

  By Gronwall's Lemma
  \begin{align*}
    \|S(s_2,s_1;\o)x\|_H^2
    &\le 2 \left(\|Z(s_2,s_1;\o)x\|_H^2 + \|N_{s_2}(\o)\|_H^2 \right) \\
    &\le 2(\|x\|_H^2 e^{- \frac{\eta}{2} (s_2-s_1)} + \int_{s_1}^{s_2} e^{- \eta (s_2-r)} C_\eta(r) \d r  + \|N_{s_2}(\o)\|_H^2).
  \end{align*}

  For  $s_1 \le s_2 \le 0$ we conclude that
  \begin{align*}
    &\|S(t,s_1;\o)x- S(t,s_2;\o)y\|_H^2 \\
    &\hskip15pt \le  2\left(\|S(s_2,s_1;\o)x\|_H^2 + \|y\|_H^2\right) e^{-\l (t-s_2)} \\
    &\hskip15pt \le 4 \left( \|x\|_H^2 e^{- \frac{\eta}{2} (s_2-s_1)} + \int_{s_1}^{s_2} e^{- \frac{\eta}{2} (s_2-r)} C_\eta(r) \d r  + \|N_{s_2}(\o)\|_H^2 + \frac{1}{2} \|y\|_H^2 \right) e^{-\l (t-s_2)}\\
    &\hskip15pt \le 4 \left( \|x\|_H^2 e^{ \frac{\eta}{2} s_1 } e^{(\l - \frac{\eta}{2})s_2} + e^{(\l - \frac{\eta}{2})s_2} \int_{s_1}^{s_2} e^{\frac{\eta}{2}r} C_\eta(r) \d r  + e^{\l s_2}\|N_{s_2}(\o)\|_H^2 + \frac{e^{\l s_2}}{2} \|y\|_H^2 \right) e^{-\l t}\\
    &\hskip15pt \le 4 \left(e^{(\l - \frac{\eta}{2})s_2} \|x\|_H^2 + e^{(\l - \frac{\eta}{2})s_2} K_\eta + e^{\l s_2}\|N_{s_2}(\o)\|_H^2 + \frac{e^{\l s_2}}{2} \|y\|_H^2 \right) e^{-\l t}\\
    &\hskip15pt \rightarrow 0\ \text{ as }\ s_1, s_2 \rightarrow - \infty,
  \end{align*}
  where $ K_\eta = \int_{-\infty}^{0} e^{\frac{\eta}{2} r} C_\eta(r) \d r$ is finite by $(S4)$,
i.e. by the subexponential growth of $\|N_t\|_V$.

Therefore, for all $t \in \R$ and $\o \in \O$ there exists a limit $\eta_t(\o)$ (independent of $x$) such that
$$\lim\limits_{s \to -\infty} S(t,s;\o)x = \eta_t(\o)$$
 holds  uniformly  in $x$ on any balls (w.r.t.  $\|\cdot\|_H$).
\end{proof}
Now we can finish the proof of Theorem $\ref{thm:single_point}$.

\noindent\textbf{Proof of Theorem $\ref{thm:single_point}$} By Lemmas
\ref{lemma:contraction} and \ref{lemma:contraction2}  we may define
    \[  \mcA(\o) = \{\eta_0(\o)\}. \]
  We shall show that this defines a global random attractor for the RDS associated with (\ref{equation 6.1}).

Since $\eta_0(\o)$ is measurable, $\mcA(\o)$ is a random compact
set. Hence we only need to check the invariance and attraction properties for
$\mcA(\o)$.

The  continuity of $x \mapsto S(t,0;\o)x$ and the flow property imply that
  \begin{align*}
    \vp(t,\o)\mcA(\o)
      &= \left\{ S(t,0;\o) \lim_{s \rightarrow -\infty}S(0,s;\o)x \right\} = \left\{ \lim_{s \rightarrow -\infty} S(t,s;\o)x \right\} \\
      &= \left\{ \lim_{s \rightarrow -\infty} S(0,s-t,\t_t\o)x \right\} = \{ \eta_0(\t_t\o) \} = \mcA(\t_t\o),
\ t>0, x\in H.
  \end{align*}
 Since the convergence in Lemmas  \ref{lemma:contraction} is uniform (locally uniform resp. in Lemma \ref{lemma:contraction2})
  with respect to $x \in H$, for any bounded set $B \subseteq H$ we have
  \begin{align*}
      d(\vp(t,\t_{-t}\o)B,\mcA(\o))
          &= \sup_{x \in B} \|S(t,0,\t_{-t}\o)x - \eta_0(\o)\|_H \\
          &= \sup_{x \in B} \|S(0,-t;\o)x - \eta_0(\o)\|_H \rightarrow 0 (t\rightarrow \infty),
\end{align*}
i.e.  $\mcA(\o)$ attracts all deterministic bounded sets.

Therefore,  $\mcA$ is a global random attractor for the RDS associated with (\ref{equation 6.1}).

 We now deduce the  unique existence of an invariant random measure $\mu_\cdot \in \mcP_\O(H)$.
For the notion of an invariant random measure we refer to
\cite[Definition 4.1]{CF94}. By \cite[Corollary 4.4]{CF94} the existence of a random attractor implies the existence of an invariant random measure. Moreover, by \cite[Theorem 2.12]{c} every invariant measure for $\vp$ is supported by $\mcA=\{ \eta_0 \}$, i.e. $\mu_\o(\{ \eta_0(\o) \}) = 1$
for $\P$-a.a. $\omega$.

  The bounds on the speed of attraction follow immediately from the respective bounds in Lemma \ref{lemma:contraction} and \ref{lemma:contraction2}.
\qed


\section{Applications to concrete SPDE}\label{sec:applications}
In this section we present several examples of admissible random
perturbations $N_t$ and also show that $(H1)-(H5)$ and $(H2^\prime)$
can be verified for many concrete SPDE. Hence Theorems  \ref{T6.1}
and \ref{thm:single_point} can be applied to show the existence of
a random attractor for those examples.

We will first show that all c\`{a}dl\`{a}g processes with stationary increments satisfy $(S1)$-$(S3)$ and
 thus Theorems  \ref{thm:generation_rds}  and  \ref{thm:single_point} are applicable.
 Of course, this contains all L\'evy processes as well as fractional Brownian motion.

\begin{lem}\label{lem:gen_metric_dns}
  Let $(N_t)_{t \in \R}$ be a $V$-valued process with stationary increments and a.s.\ c\`{a}dl\`{a}g paths. Then there is a metric dynamical system $(\O, \mcF, \P, \t_t)$ and a version $\tilde{N}_t$ (cf. \cite[Definition 1.6]{RY99}) on $(\O, \mcF, \P, \t_t)$ such that $\tilde{N}_t$ satisfies $(S1)$-$(S3)$.
\end{lem}
\begin{proof}
  We choose $\O = D(\R; V)$ to be the set of all c\`{a}dl\`{a}g
  functions endowed with the Skorohod topology (cf. \cite{A98}, pp. 545),
$\mcF = \mcB(\O)$, $\t_t(\o)=\o(t+\cdot)-\o(t)$ and $\P = \mcL(N)$ to be the law of $N_t$
(or more precisely its restriction on $\O$). Note that $\mcF$ is the trace in $\O$ of
the product $\sigma$-algebra $\mcB(V)^\R$ and $(t,\o) \mapsto \t_t(\o)$ is measurable. Since $N_t$ has stationary increments we know that $\t_t\P = \P$. Hence $(\O,\mcF,\P,\t_t)$
defines a metric dynamical system and the coordinate process $\tilde{N}_t$ on $\O$ is a version of $N_t$ satisfying $(S1),(S2)$ and $(S3)$.
\end{proof}

We will prove the asymptotic bound $(S4)$ for two classes of
processes. The first class consists of all processes with independent
increments (e.g. L\'evy processes) where the proof is based on the
strong law of large numbers, and the second class consists of all processes
with H\"older continuous paths (e.g. fractional Brownian motion),
for which we use Kolmogorov's continuity theorem and the dichotomy
of linear growth for stationary processes.

\begin{lem}\label{len:levy}
  Let $V$ be a separable Banach space and $N_t$ be a $V$-valued L\'evy process with L\'evy characteristics
  $(m,R,\nu)$ (e.g. cf. \cite[Corollary 4.59]{PZ07}). Assume that $\int_V \left(\|x\|_V \vee \|x\|^2_V \right) d\nu(x) < \infty$, then we have
  $\P$-a.s.
    \[ \frac{N_t}{|t|} \rightarrow \oldpm \E N_1 \ (t \rightarrow \oldpm \infty).\]
\end{lem}
\begin{proof}
  Since $\bar{N}_t := N_{-t}$  is also a L\'evy process satisfying the assumptions and $\E \bar{N}_1 = - \E N_1$,
  it is sufficient to prove the assertion for $t \to +\infty$.
  By the L\'evy-It\^o decomposition for Banach space-valued L\'evy processes (cf. \cite[Theorem 4.1]{A07}) we have
    \[ N_t = mt + W_t + \int_{B_1(0)} x \tilde{N}(t,dx) + \int_{B_1^c(0)} x N(t,dx), \]
  where $m \in V$, $W_t$ is a $V$-valued Wiener process and
  \begin{align}\label{eqn:comp_part}
    \int_{B_1(0)} x \tilde{N}(t,dx)
    := \lim_{n \to \infty } \int_{\{ \frac{1}{n+1} \le ||x||_V < 1 \}} x \tilde{N}(t,dx)
    = \lim_{n \to \infty} \sum_{k=1}^n \int_{\{ \frac{1}{k+1} \le ||x||_V < \frac{1}{k} \}} x \tilde{N}(t,dx) 
  \end{align}
  is a $\P$-$a.s.$ limit of compensated compound Poisson processes.

By an analogous calculation to \cite[pp. 49]{PZ07} we
  have
  \begin{align*}
     \E \| \int_{\{ \epsilon \le \|x\|_V < 1 \}} x \tilde{N}(t,dx) \|_V^2
     &\le 2 t \int_{\{ \epsilon \le \|x\|_V < 1 \}} \|x\|_V^2 d\nu(x) + 4 \left( t  \int_{\{ \epsilon \le \|x\|_V < 1 \}} \|x\|_V d\nu(x)
     \right)^2,
  \end{align*}
  and
  \begin{align*}
     \E \| \int_{B_1^c(0)} x N(t,dx) \|_V^2 &\le t \int_{B_1^c(0)} \|x\|_V^2 d\nu(x) + \left( t \int_{B_1^c(0)} \|x\|_V d\nu(x) \right)^2.
  \end{align*}
  Thus,
  \begin{align}\label{eqn:comp_bound}
     & \sup_{n} \E \| \sum_{k=1}^n \int\limits_{\{ \frac{1}{k+1} \le \|x\|_V < \frac{1}{k} \}} x \tilde{N}(t,dx)
     \|^2 = \sup_{n} \E \| \int\limits_{\{ \frac{1}{n+1} \le \|x\|_V < 1 \}} x \tilde{N}(t,dx) \|^2 \\
     \le& 2t \int_{B_1(0)} \|x\|_V^2 d\nu(x) + 4 \left( t  \int_{B_1(0)} \|x\|_V d\nu(x) \right)^2 < \infty.\nonumber
  \end{align}
  By \eqref{eqn:comp_part} $\int_{B_1(0)} x \tilde{N}(t,dx)$ is the limit of a $\P$-$a.s.$
  converging series of independent random variables and by \cite[Theorem 3.4.2]{L86}
  the bound \eqref{eqn:comp_bound} implies that the convergence in \eqref{eqn:comp_part} also holds in $L^2(\O;V)$. Hence $N_t \in L^2(\O;V)$ and
    \[ \E N_t = t \left(m +  \int_{B_1^c(0)} x d\nu(x) \right) = t \E N_1. \]
 Let now $N_t$ be centered (i.e. $\E N_t = 0$) and note
   \[ N_n = N_n - N_{n-1} + N_{n-1} - N_{n-2} + ... + N_1 ,\]
  then by the law of large numbers for Banach space-valued random vectors (cf. \cite[Theorem III.1.1]{H77})
we have  $\P$-$a.s.$
   \[ \frac{N_n}{n} \to \E[N_1] = 0\   (n\rightarrow\infty) . \]
  It remains to derive the  bound for $N_t - N_{[t]}$. Let
  $$S_n := \sup_{s \in [0,1]} ||N_{n+s} - N_n||_V.$$
   Since $N_t$ is centered and has first moment, it is a martingale. Thus $\|N_t\|_V$ is a non-negative c\`{a}dl\`{a}g submartingale
   and  Doob's maximal inequality implies that
    \[ \E S_0 = \E \sup_{s \in [0,1]} ||N_s\|_V \le 2 (\E \|N_1\|_V^2)^\frac{1}{2} < \infty.\]
  Since  $S_n$ are i.i.d., by the strong law of large numbers we have for $N\rightarrow\infty$
    \[ \frac{\sum_{n=1}^{N} S_n}{N} \to \E[S_1],\ \P\text{-a.s.} . \]
  In particular, we have $\frac{S_N}{N} \to 0$, $\P$-$a.s.$. Consequently,
   \begin{align*}
      \frac{\|N_t\|_V}{t}&
      \le \frac{[t]}{t} \left( \frac{\|N_t - N_{[t]}\|_V}{[t]} +  \frac{\|N_{[t]}\|_V}{[t]}
      \right)\\
      &\le \frac{[t]}{t} \left( \frac{S_{[t]}}{[t]} + \frac{\|N_{[t]}\|_V}{[t]} \right) \to 0\   (t\to \infty),\ \P\text{-a.s.}\ .
   \end{align*}

   For $N_t$ not necessarily centered we have
   \begin{align*}
     \frac{N_t}{t}
     = \frac{N_t-\E N_t}{t} + \frac{\E N_t}{t}
     = \frac{N_t-\E N_t}{t} + \E N_1 \to \E N_1,\ (t\to \infty),\ \P\text{-a.s.}\ .
   \end{align*}
\end{proof}

We now prove an asymptotic bound for processes satisfying the assumptions of Kolmogorov's continuity theorem.
The proof is similar to \cite[Lemmas 2.4 and 2.6]{MS04} where the case of fractional Brownian motion
with Hurst parameter $H\in(\frac{1}{2},1)$ is considered.
 However, note that we do not require $\g = 2$ in (\ref{eqn:kol_assumpt}), hence here we can include
fractional Brownian motion
with any Hurst parameter $H\in(0,1)$ (see Lemma \ref{lem:fBM}).

\begin{lem}\label{lem:growth_bd_1}
  Let $(N_t)_{t \in \R}$ be a process on a metric dynamical system $(\O,\F,\P,\t_t)$ with values in a Banach space $V$ such that $(S1)$ holds.
  Assume that there exist constants $\gamma > 1$, $\a > 0$ and $C \in \R$ such that
  \begin{equation}\label{eqn:kol_assumpt}
    \E \|N_t-N_s\|_V^\g \le C |t-s|^{1+\a},\ \forall t,s \in \R.
  \end{equation}
  Then there exists a $\t_t$-invariant set $\O_0 \subseteq \O$ with $\P(\O_0) = 1$ and for any
   $\epsilon > 0$, $\o \in \O_0$, $0 < \b < \frac{\a}{\g}$ and any interval
    $[s_0,t_0] \subseteq \R$ there exist constants $C_1 = C_1(\epsilon,\o,\b), C_2=C_2(\o,\b,s_0,t_0)>0$
    such that
    \begin{align*}
       &\|N_t(\o)\|_V \le \epsilon |t|^2 + C_1,\ \forall t\in\R  \text{ and} \\
       &\|N_\cdot(\o)\|_{C^\b([s_0,t_0];V)} \le C_2.
    \end{align*}
    In particular, $N_t$ satisfies $(S4)$.
\end{lem}
\begin{proof}
  Since $\bar{N}_t := N_{-t}$ also satisfies the assumptions,
  it is enough to prove the assertion for $t > 0$.
  Firstly, we have
  \begin{align}\label{eqn:N-bound}
    \|N_t(\o)\|_V
    &= \|N_t(\o)-N_{[t]}(\o)+ N_{[t]}(\o) - N_{[t]-1}(\o) + ... + N_1(\o)-N_0(\o)+N_0(\o)\|_V  \nonumber\\
    &= \|N_{t-[t]}(\t_{[t]}\o)-N_0(\t_{[t]}\o) + N_1(\t_{[t]-1}\o) - N_0(\t_{[t]-1}\o) + ... \nonumber\\
      &\hskip15pt + N_1(\o)-N_0(\o)+N_0(\o) \|_V \\
    &\le \|N_\cdot(\t_{[t]}\o)\|_{C^\b([0,1];V)} + \|N_\cdot(\t_{[t]-1}\o)\|_{C^\b([0,1];V)} + ... +\|N_\cdot(\o)\|_{C^\b([0,1];V)} \nonumber\\
      &\hskip15pt +\|N_0(\o)\|_V,\ t\in\R_+. \nonumber
  \end{align}
  Hence we need to derive a bound for $\|N_\cdot(\t_r\o)\|_{C^\b([0,1];V)}$ as a function of $r$.

  Using Kolmogorov's continuity theorem (cf. \cite[Theorem 1.4.1]{K90}) and \eqref{eqn:kol_assumpt} we obtain
    \[ \|N_\cdot(\o)\|_{C^\b([s,t];V)} \le K(\o,\b,s,t) \in L^\gamma(\O,\mcF^\P,\P) \subseteq L^1(\O,\mcF^\P,\P),\ \forall\ s<t, \]
  where $\mcF^\P$ is the completion of $\mcF$ with respect to $\P$. Note that
  \begin{align}\label{eqn:hoelder_bd}
    \|N_\cdot(\t_{r}\o)\|_{C^\b([s,t];V)}
     &= \sup_{u \ne v, u,v \in [s,t]} \frac{\|N_u(\t_{r}\o)-N_v(\t_{r}\o)\|_V}{|u-v|^\b} \\
     &= \sup_{u \ne v, u,v \in [s,t]} \frac{\|N_{u+r}(\o)-N_{v+r}(\o)\|_V}{|u-v|^\b}\nonumber \\
     & = \|N_\cdot(\o)\|_{C^\b([s+r,t+r];V)}, \nonumber
  \end{align}
  for all $s < t$ and $r \in \R$. Hence
    \[ \sup_{r \in [0,1]} \|N_\cdot(\t_r\o)\|_{C^\b([s,t];V)} \le \|N_\cdot(\o)\|_{C^\b([s,t+1];V)} \in L^1(\O),\ \forall\ s<t.\]
  The dichotomy of linear growth for stationary processes (cf. \cite[Proposition 4.1.3 (ii)]{A98}) states that any measurable map $f: \O \to \R$ on a metric dynamical system $(\O,\mcF,\P,(\t_t)_{t \in \R})$ with $\sup_{t \in [0,1]} f^+(\t_t \cdot) \in L^1(\O)$ grows sublinearly, i.e.
    \[ \limsup_{t \to \oldpm \infty} \frac{1}{|t|}f(\t_t \o) = 0, \]
  on an invariant set of full $\P$ measure. We conclude that there is a $\t_t$-invariant set $\O_0 \subseteq \O$ with $\P(\O_0) = 1$ such that
    \[ \lim_{|t| \to \infty} \frac{1}{|t|} \|N_\cdot(\t_t\o)\|_{C^\b([0,1];V)} = 0,\ \forall\ \o \in \O_0.\]
   Hence for every $\epsilon > 0$, $\o \in \O_0$ there exists a constant $T:=T(\epsilon, \o) \in \N$ such that
   $$\|N_\cdot(\t_t\o)\|_{C^\b([0,1];V)} \le \epsilon |t|,\  |t| \ge T.$$
    By \eqref{eqn:N-bound} this implies
  \begin{align*}
     \|N_t(\o)\|_V
      &\le \sum_{k=T}^{[t]} \|N_\cdot(\t_{k}\o)\|_{C^\b([0,1];V)} + \sum_{k=0}^{T-1} \|N_\cdot(\t_{k}\o)\|_{C^\b([0,1];V)} +\|N_0(\o)\|_V  \\
      &\le \epsilon [t]^2 + T \|N_\cdot(\o)\|_{C^\b([0,T];V)} +\|N_0(\o)\|_V  \\
      &\le \epsilon [t]^2 + T K(\o,\b,0,T) +\|N_0(\o)\|_V, \ t \in
      \R_+.
  \end{align*}

\end{proof}

\begin{cor}\label{cor1}
  Let $(N_t)_{t \in \R}$ be a $V$-valued process with stationary increments and a.s.\ c\`{a}dl\`{a}g paths. Assume that \eqref{eqn:kol_assumpt} or the assumptions of Lemma \ref{len:levy} hold, then there is a metric dynamical system $(\O, \mcF, \P, \t_t)$ and a version $\tilde{N}_t$ on $(\O, \mcF, \P, \t_t)$ such that $\tilde{N}_t$ satisfies $(S1)$-$(S4)$.
\end{cor}

Now we show that (\ref{eqn:kol_assumpt}) holds for fractional
Brownian Motion (fBM) with any Hurst parameter. We first recall the definition of Banach
space-valued fBM.
\begin{defn}[Fractional Brownian motion]
  Let $H \in (0,1)$ and $R: V^* \to V$ be a bounded linear and non-negative  symmetric operator.
A $V$-valued $\P$-a.s.\ continuous centered Gaussian process $B_t^H$ starting at $0$ is called an $R$-fBM with Hurst parameter $H$ if the covariance is given by
         \[ \E[ \ _{V^*}\<x,B_t^H\>_V \ _{V^*}\<y,B_s^H\>_V ]= \frac{1}{2}(t^{2H}+s^{2H}-|t-s|^{2H}) \ _{V^*}\<x,Ry\>_V\]
  for all $x,y \in V^*$, $t,s \in \R_+$.
\end{defn}

It is easy to see that $B_t^H$ has stationary increments. Thus, according to Lemma \ref{lem:gen_metric_dns} we will always consider
the canonical realization of fBM in this paper.

Let
  \[V \subseteq H \subseteq V^*\]
be a Gelfand triple, $B_t^H$ be an $R$-fBM in $H$ and $\l_k \ge 0$, $e_k \in H$ such that $Re_k = \l_k e_k$. Then $B_t^H$ has the representation
\begin{equation}\label{eqn:fBM-rep}
  B_t^H = \sum_{k=1}^\infty \sqrt{\l_k} \b_k^H(t) e_k,
\end{equation}
where $\b_k^H$ are independent real-valued fBM and the convergence holds $\P$-a.s.\ as well as in each $L^p(\O;H)$.

\begin{lem}\label{lem:fBM}
  Let $B_t^H$ be a fBM in $H$ with representation \eqref{eqn:fBM-rep} and assume that $K = \sum_{k=1}^\infty \sqrt{\l_k} \|e_k\|_V < \infty$.
  Then $B_t^H$ satisfies (\ref{eqn:kol_assumpt}),  more precisely,
     for each $m \in \N$ there is a constant $C > 0$ such that
  \[ \E\ \|B_t^H - B_s^H\|_V^{2m} \le CK |t-s|^{2Hm}, \ s,t\in\R .\]
\end{lem}
\begin{proof}
  By the comparability of Gaussian moments (cf. \cite[Corollary 3.2]{LT91}) we have
    \[ \left[\E \|B_t^H - B_s^H\|_V^{2m} \right]^{\frac{1}{2m}}
    = \left[\E \|B_{t-s}^H\|_V^{2m} \right]^{\frac{1}{2m}} \le C \E \|B_{t-s}^H\|_V,\]
  where $C > 0$ is a constant depending only on $m$.

  By our assumption we know that the convergence in \eqref{eqn:fBM-rep} also holds in $L^1(\O;V)$.
  Hence we have
  \begin{align*}
   \left[ \E \|B_t^H - B_s^H\|_V^{2m} \right]^\frac{1}{2m}
    &\le C \E \|B_{t-s}^H\|_V \\
     &= C \lim_{N \to \infty} \E \|\sum_{k=1}^N \sqrt{\l_k} \b_k^H(t-s) e_k\|_V  \\
    &\le C \lim_{N \to \infty} \sum_{k=1}^N \sqrt{\l_k} \|e_k\|_V \E |\b_k^H(t-s)| \\
    & \le CK |t-s|^{H}, \ s,t\in\R.
  \end{align*}
 In particular,  choosing $m$ such that $2Hm > 1$ we get \eqref{eqn:kol_assumpt}.
\end{proof}

We now proceed to examples of SPDE satisfying $(H1)-(H5)$ and $(H2^\prime)$.
 Note that most of those assumptions are well known and have been
 used extensively in recent years for investigating SPDE within the variational framework, e.g. see $(H1)-(H4)$ in \cite{GM07,GM09,L09,L10-2,L10,PR07} and $(H2^\prime)$ in \cite{DRRW06,L09,L10}. It has  also been proved that (\ref{condition 1}) in $(H5)$ holds for many SPDE in \cite{L10-2}. Hence, we only need to verify (\ref{noise condition}) in $(H5)$.

The following elementary lemma is crucial for verifying $(H2^\prime)$ (cf. \cite{L09,L10}).
 For the proof see e.g. \cite{L10}.
\beg{lem}\label{lem5} Let $(E, \langle\cdot,\cdot\rangle)$ be a
 Hilbert space and $\|\cdot\|$ denote its norm.  Then
for any $r\geq0$ we have \beq\label{3.1} \langle\|a\|^ra-\|b\|^rb,
a-b\rangle\geq2^{-r}\|a-b\|^{r+2}, \ a,b\in E. \end{equation}
\end{lem}

\beg{exa} Let $\Lambda$ be an open bounded domain in $\mathbb{R}^d$
 and $L^p:=L^p(\Lambda)$ for some fixed  $p\geq 2$.
 Consider
the following triple
$$V:=L^{p}\subseteq  H:= L^2\subseteq (L^{p})^*\equiv L^{\frac{p}{p-1}}$$
and the stochastic equation
 \beq\label{e6.6}
 \d X_t=f(X_t) \d t+ \d N_t,\ t\in \mathbb{R},
  \end{equation}
where $N_t$ is an $L^p$-valued process with stationary increments and a.s.\ c\`{a}dl\`{a}g paths,
$f:\mathbb{R}^d\rightarrow \mathbb{R}^d$ is continuous and satisfies the following conditions:
\begin{align}\label{coercive}
&\<f(x)-f(y), x-y \>\le -\l |x-y|^\beta;\nonumber \\
& \<f(x), x \>\le -\delta|x|^p+K|x|^2+C;\\
& |f(x)|\le C(|x|^{p-1}+1),\ x,y\in \mathbb{R}^d,\nonumber
\end{align}
where $C,\l>0,\delta>0,\beta>2$ are some constants and $\<\cdot,\cdot\>$ is the inner product on $\mathbb{R}^d$.
Then
the RDS generated by $(\ref{e6.6})$ has a unique random fixed point and the other assertions in Theorem \ref{thm:single_point} also hold.
 If $\beta=2$ in (\ref{coercive}), then the conclusions still hold,
provided $N_t$ also satisfies $(S4)$.
 \end{exa}

\beg{proof} Using a similar argument as in \cite[Example 4.1.5]{PR07}, one can show that $(H1)$,$(H2^\prime)$,
$(H3)$ and $(H4)$ hold for (\ref{e6.6}) with $\alpha=p$.
 Hence Theorem \ref{thm:single_point} applies.
\end{proof}

\begin{rem} (i) A typical example for $f$ is as follows (cf. \cite{L10-2,L10,PR07}):
 $$f(x)=-|x|^{p-2}x+ \eta x,\ \eta\le 0.$$

(ii) The first inequality in (\ref{coercive}) implies that
$$ \<f(x), x \>\le -\frac{\lambda}{2}  |x|^\beta +C.$$
Therefore, if $\beta\ge p$, then the second inequality (so-called coercivity condition) in (\ref{coercive}) automatically holds.

(iii) If $N_t$ is a finite-dimensional fBM,  the existence of a random fixed point for (\ref{e6.6})
  has also been studied in \cite{GKN09}.
Compared with the  result in \cite{GKN09}, we only require a coercivity condition (the second inequality in  (\ref{coercive}))
on $f$ instead of
assuming $f$ to be continuously differentiable as in \cite{GKN09}.
Another improvement is that we can allow equation (\ref{e6.6}) to be driven by
infinitely many fractional Brownian motions or by L\'{e}vy noise.
\end{rem}

\beg{exa} (Stochastic reaction-diffusion equation)\label{exa:srde}\\
Let $\Lambda$ be an open bounded domain in $\mathbb{R}^d$. We consider
the following triple
$$ V:=W^{1,2}_0(\Lambda) \subseteq L^2(\Lambda)\subseteq (W^{1,2}_0(\Lambda))^*$$
and the stochastic reaction-diffusion equation
 \beq\label{rd}
 \d X_t=(\Delta X_t-|X_t|^{p-2}X_t+\eta X_t)\d t+\d N_t,
  \end{equation}
where $1 \le p \le 2$ and $\eta$ are some constants,  $N_t$ is a $V$-valued process with stationary increments and a.s.\ c\`{a}dl\`{a}g paths.
\begin{enumerate}
 \item[(1)] If $\eta\le 0$ and $(S4)$ holds, then all assertions in Theorem \ref{thm:single_point} hold for $(\ref{rd})$ with $\b = 2$.
 \item[(2)] If  $\eta> 0$, $N_\cdot(\omega) \in L^2([-1,0]; W^{3,2}(\Lambda))$ for $\P$-$a.e.\ \omega$ and
 satisfies $(S4)$, then the stochastic flow associated with $(\ref{rd})$ has a compact random attractor.
\end{enumerate}
 \end{exa}

\beg{proof}
(1) By Lemma \ref{lem:gen_metric_dns} we know that $(S1)$-$(S3)$ hold. It is also well known that
$(H1)$-$(H4)$ hold for (\ref{rd}) (cf. \cite{L09,L10,PR07}).
If $\eta\le 0$, then it is easy to show that $(H2^\prime)$ holds with $\beta=2$. Therefore, all assertions in Theorem \ref{thm:single_point}  hold for $(\ref{rd})$.

(2) According to Theorem \ref{T6.1} one only needs to verify $(H5)$. Let $S=W^{1,2}_0(\Lambda)$ and
 $\Delta$
be the Laplace operator on $L^2(\Lambda)$ with Dirichlet boundary conditions.
We define
  $$T_n=-\Delta(1-\frac{\Delta}{n})^{-1}.$$
Let $\{P_t\}_{t\geq 0}$ and $\mathcal{E}$ denote the semigroup and Dirichlet form corresponding to $\Delta$.
It is easy to show that $T_n$ are  continuous operators on
$W^{1,2}_0(\Lambda)$ by noting that
$$T_n=n\left(I-\left(I-\frac{\Delta}{n}\right)^{-1}    \right).   $$
Then we have
 \ce { }_{V^*}\langle \Delta u, T_n u\rangle_{V}
 &=& { }_{V^*}\langle \Delta u, -\Delta(1-\frac{\Delta}{n})^{-1}u
 \rangle_{V}\\
 &=&{ }_{V^*}\langle \Delta u, nu-n(1-\frac{\Delta}{n})^{-1}u
 \rangle_{V}\\
 &=&-n\int_0^\infty e^{-t}\langle \nabla u, \nabla u-\nabla P_{\frac{t}{n}} u
 \rangle_{L^2(\Lambda)} \d t\\
 &\leq&-n\int_0^\infty e^{-t}(\mathcal{E}(u, u)-\mathcal{E}(u, P_{\frac{t}{n}}u)) \d t\\
 &\leq& 0,
\de
where the last step follows from the contraction property of the  Dirichlet form $\mathcal{E}$.

By using a similar argument one can show that
$$ {  }_{V^*}\langle -|u|^{p-2}u+\eta u, T_n u
 \rangle_{V}\le \eta\|u\|_n^2,\  u\in W^{1,2}_0(\Lambda).$$
Hence $(\ref{condition 1})$ holds. Using the fact that $P_t$ is bounded on $W^{1,2}_0(\Lambda)$ and
$N_\cdot(\omega) \in L^2([-1,0]; W^{3,2}(\Lambda))$ for $\P$-$a.e.\ \omega$ we have
\begin{equation}\label{e2}
\begin{split}
 \int_{-1}^0 \|T_n N_t\|_V^2 dt &=  \int_{-1}^0 \|-\Delta(I-\frac{\Delta}{n})^{-1} N_t\|_V^2 dt \\
 &= \int_{-1}^0 \|(I-\frac{\Delta}{n})^{-1}(\Delta N_t)\|_V^2 dt \\
& \le C \int_{-1}^0 \|\Delta N_t\|_V^2 dt < \infty,
\end{split}\end{equation}
where the third step follows from  the following formula
$$    (I-\frac{\Delta}{n})^{-1}v=\int_0^\infty e^{-t}P_{\frac{t}{n}}v \d t, \  v\in V.      $$
Hence (\ref{noise condition}) holds.  Then the existence
of the random attractor for (\ref{rd}) follows from Theorem
\ref{T6.1}.
\end{proof}

\begin{rem}\label{rmk:RDE}
  In Example \ref{exa:srde} we had to restrict to reaction terms of at most linear growth.
This restriction is due to the fact that the variational approach to SPDE as presented in \cite{KR79,PR07} does
not apply to nonlinearities of arbitrary high order. However, we only used the results from \cite{KR79,PR07} to
 construct the associated RDS. Therefore, as soon as we can obtain the corresponding RDS by some other method,
 our arguments can be used without change to prove the existence of the random attractor.
More precisely, let $\Lambda$ be an open bounded domain in $\mathbb{R}^d$. We consider the following triple
    $$ V:=W^{1,2}_0(\Lambda) \cap L^p(\Lambda) \subseteq H:= L^2(\Lambda)\subseteq (W^{1,2}_0(\Lambda)\cap L^p(\Lambda))^*$$
  and the stochastic reaction-diffusion equation
  \beq\label{rd2}
    \d X_t=(\Delta X_t-|X_t|^{p-2}X_t+\eta X_t)\d t+\d N_t,
  \end{equation}
  where $2 < p$, $\eta \in \R$ are some constants and $N_t$ is a $V$-valued process with
 stationary increments and a.s.\ c\`{a}dl\`{a}g paths. Note that \eqref{rd2} does not satisfy $(H3)$-$(H4)$
with the same parameter $\a$. Nevertheless, the associated RDS can be defined by an analogous
transformation into a random PDE. The existence and uniqueness of solutions
 for the transformed equation \eqref{e1} follows by a standard proof via Galerkin approximations
(cf. \cite[pp. 91]{T97}). The proof of condition $(H5)$ carries over without change.
If $\eta \le 0$ then Theorem \ref{thm:single_point} can be applied with $\b = p$.
If $\eta > 0$, $N_\cdot (\omega) \in L^2([-1,0]; W_0^{3,2}(\L)) \cap L^p([-1,0]; W_0^{2,p}(\L))$
for $\P$-$a.e.\ \omega$ and satisfies $(S4)$,
 then the same arguments as for Theorem \ref{T6.1} yield the existence of the random attractor.

In \cite[Section 5]{CF94} the existence of a random attractor for stochastic reaction-diffusion equations 
perturbed by finite-dimensional Brownian noise is obtained under the assumption that the noise takes values in
 $H^2(\L) \cap H^1_0(\L) \cap W^{2,\frac{p}{p-1}}(\L)$. In comparison, we can allow infinite-dimensional noise and include
 fractional Brownian motion as well as L\'evy type noise, but we need to require slightly more regular noise taking values in $H^{3}(\L) \cap H_0^1(\L) \cap W^{2,p}(\L)$.
\end{rem}

\begin{rem} Simple examples of noises satisfying the assumptions are given by finite-dimensional noise. Let $N \in \mathbb{N}$ and
\begin{equation}\label{explicit noise}
 N_t=\Sum_{n=1}^N \varphi_n \beta^H_n(t)
\left(\text{or}~~ N_t=\Sum_{n=1}^N \varphi_n L_n(t)~ \right), \  t\in \mathbb{R},
\end{equation}
where $\varphi_n \in W^{3,2}_0(\Lambda) \cap W^{2,p}_0(\L)$ and $\beta^H_n$ are independent two-sided fractional Brownian motions with Hurst parameter $H\in(0,1)$ (or $L_n$ are independent two-sided L\'{e}vy processes). It is easy to show that the noise (\ref{explicit noise}) satisfies all assumptions required in the above example.
Noise of this form can also be used for those examples below as well by choosing appropriate spaces for $\varphi_n$.
\end{rem}

\begin{exa} (stochastic porous media  equation)\\
Let $\Lambda$ be an open bounded domain in $\mathbb{R}^d$. For  $r>1$ we  consider the following triple
$$V:=L^{r+1}(\Lambda)\subseteq H:=W^{-1,2}_0(\Lambda) \subseteq V^*$$
and the stochastic porous media  equation
 \begin{equation}\label{PME}
 \d X_t=\left(\Delta(|X_t|^{r-1}X_t)+\eta X_t\right)\d t+ \d N_t,
  \end{equation}
where $\eta$ is a constant,  $N_t$ is  a $V$-valued process with stationary increments and a.s.\  c\`{a}dl\`{a}g paths.
\begin{enumerate}
 \item[(1)] If $N_\cdot (\omega) \in L^{r+1}\left([-1,0]; W^{2,r+1}(\Lambda)\right)$ for $\P$-$a.e.\ \omega$ and
 satisfies $(S4)$, then the stochastic flow associated with $(\ref{PME})$ has a compact random attractor.
 \item[(2)] If $\eta\le 0$, then all assertions in Theorem \ref{thm:single_point}  hold for $(\ref{PME})$.
\end{enumerate}
 \end{exa}

\begin{proof} (1)
 According to \cite[Example 4.1.11; Remark 4.1.15]{PR07} we know that $(H1)$-$(H4)$ hold for (\ref{PME}). By Lemma \ref{lem:gen_metric_dns} and the assumptions we know $(S1)$-$(S4)$ also hold. Hence we only need to verify  $(H5)$ in Theorem \ref{T6.1}.

 Let $S=L^2(\Lambda)$ and $\Delta$ be the Laplace operator on $L^2(\Lambda)$ with Dirichlet boundary conditions.
We define
$$T_n=-\Delta(I-\frac{\Delta}{n})^{-1}=n\left(I-(I-\frac{\Delta}{n})^{-1}\right).$$
It is well known that  the heat semigroup $\{P_t\}$ (generated by $\Delta$) is contractive on  $L^p(\Lambda)$ for any $p>1$.
Then by the same argument as in (\ref{e2})
 we know that (\ref{noise condition}) holds.

In order to show that $T_n$ are
  continuous operators on $L^{r+1}(\Lambda)$ we use the formula
$$ (I-\frac{\Delta}{n})^{-1}u=\int_0^\infty e^{-t}P_{\frac{t}{n}}u \d t.  $$
By  H\"{o}lder's inequality
 and the contractivity of  $\{P_t\}$ on $L^{r+1}(\Lambda)$ we have
 \ce
 && _{V^*}\langle \Delta(| u|^{r-1} u)+\eta u, -\Delta(1-\frac{\Delta}{n})^{-1}u
 \rangle_{V}\\
 &=& \langle | u|^{r-1} u, nu-n(1-\frac{\Delta}{n})^{-1}u
 \rangle_{L^2} +\eta\|u\|_n^2 \\
 &=& -n\int_0^\infty e^{-t}\left(\int_\Lambda | u|^{r+1}\d x-\int_\Lambda | u|^{r-1}u\cdot P_{\frac{t}{n}}
 u\d x\right) \d t +\eta\|u\|_n^2 \\
 &\leq& \eta\|u\|_n^2,\ \forall u\in L^{r+1}(\Lambda).
\de
Hence  (\ref{condition 1}) holds and the assertion follows from  Theorem \ref{T6.1}.

(2)
If $\eta\le 0$, then by Lemma \ref{lem5} it is easy to show that $(H2^\prime)$ holds with $\beta=r+1$ (cf. \cite{L09,L10}).
Hence all assertions in Theorem \ref{thm:single_point}  hold for $(\ref{PME})$.
 \end{proof}

\begin{rem}\label{rmk:PME}
  In \cite{BGLR10} the existence of a random attractor for generalized porous media equations perturbed by
 finite-dimensional Brownian noise has been proven under the assumption that the noise
 takes values in $W_0^{1,r+1}$. In the case of the standard porous medium equation our results thus extend 
\cite{BGLR10} to infinite-dimensional noise and fractional Brownian motion as well as L\'evy type noise, if the noise is more regular, i.e.\ takes values in $W^{2,r+1}$.
\end{rem}

\beg{exa} ( Stochastic  $p$-Laplace  equation)\\
 Let $\Lambda$ be an open bounded domain in $\mathbb{R}^d$ with convex and smooth boundary. We consider
the following triple
$$V:=W^{1,p}(\Lambda)\subseteq H:= L^2(\Lambda)\subseteq (W^{1,p}(\Lambda))^*$$
and the stochastic $p$-Laplace equation
 \beq\label{PLE}
 \d X_t=\left[ \mathbf{div}(|\nabla X_t|^{p-2}\nabla X_t)-\eta_1|X_t|^{\tilde{p}-2}X_t+\eta_2 X_t\right]\d t+\d N_t,
  \end{equation}
where $2< p<\infty, 1\leq\tilde{p}\leq p$, $\eta_1\ge 0$, $\eta_2 \in \R$ are some constants and  $N_t$ is  a $V$-valued process
with stationary increments and a.s.\  c\`{a}dl\`{a}g paths.
\begin{enumerate}
 \item[(1)] If $N_\cdot (\omega)\in L^p ([-1,0]; W^{3,p}(\Lambda))$ for $\P$-$a.e.\ \omega$ and satisfies $(S4)$, then
the stochastic flow associated with $(\ref{PLE})$ has a compact random attractor.
 \item[(2)] If $\eta_2 \le 0$, then all assertions in Theorem \ref{thm:single_point}  hold for $(\ref{PLE})$.
\end{enumerate}
\end{exa}

 \begin{proof}
(1)
 According to  \cite[Example 4.1.9]{PR07} and the assumptions, we only need to verify $(H5)$ in Theorem \ref{T6.1}.

 Let $S=W^{1,2}(\Lambda)=\mathcal{D}(\sqrt{-\Delta})$, where $\Delta$ is the
Laplace operator on $L^2(\Lambda)$ with Neumann boundary conditions. It is well known
that the corresponding semigroup $\{P_t\}$ is the
Neumann heat semigroup (i.e. the corresponding Markov process is Brownian motion with reflecting boundary conditions).
 Moreover, we know that $P_t$ maps $L^p(\Lambda)$ into $W^{1,p}(\Lambda)$ continuously (see
\cite[Section 2]{CR04} for more general results). Then for all $t\geq0$,
$P_t: W^{1,p}(\Lambda)\rightarrow W^{1,p}(\Lambda)$
is continuous.

Now we define
$$T_n=-\Delta(I-\frac{\Delta}{n})^{-1}=n\left(I-(I-\frac{\Delta}{n})^{-1}\right).$$
It is easy to show that $T_n$ are
also  continuous operators on $W^{1,p}(\Lambda)$ since
$$ (I-\frac{\Delta}{n})^{-1}u=\int_0^\infty e^{-t}P_{\frac{t}{n}}u \d t.  $$
Moreover, since the boundary of the domain is convex and smooth, we have the following  gradient estimate (cf. \cite[Theorem 2.5.1]{W05})
\begin{equation}\label{gradient estimate}
 |\nabla P_t u|\leq P_t |\nabla u|,\ \  u\in W^{1,p}(\Lambda).
\end{equation}
 Since  $\{P_t\}$ is a contractive semigroup on $L^p(\Lambda)$,
  it is easy to see that $\{P_t\}$ is also a contractive semigroup on $W^{1,p}(\Lambda)$ by (\ref{gradient estimate}).
Therefore,
 \ce
 && _{V^*}\langle \mathbf{div}(|\nabla u|^{p-2}\nabla u), T_n u
 \rangle_{V}\\
 &=& _{V^*}\langle \mathbf{div}(|\nabla u|^{p-2}\nabla u), nu-n(1-\frac{\Delta}{n})^{-1}u
 \rangle_V \\
 &=&n\int_0^\infty e^{-t}{ }_{V^*}\langle \mathbf{div}(|\nabla u|^{p-2}\nabla u), u-P_{\frac{t}{n}} u
 \rangle_V \d t\\
 &=& -n\int_0^\infty e^{-t}\left(\int_\Lambda|\nabla u|^{p}\d x-\int_\Lambda |\nabla u|^{p-2}\nabla u\cdot\nabla P_{\frac{t}{n}}
 u\d x\right)
  \d t\\
 &\leq&0, \ u\in W^{1,p}(\Lambda),
\de
where in the last step we used H\"{o}lder's inequality
 and the contractivity of  $\{P_t\}$ on $W^{1,p}(\Lambda)$ to conclude \ce
&&\int_\Lambda|\nabla u|^{p-2}\nabla u\cdot\nabla P_{s}
 u \d x\\
 &\leq&\left(\int_\Lambda|\nabla
u|^{p}\d x\right)^{\frac{p-1}{p}}\cdot\left(\int_\Lambda|\nabla
P_su|^p\d x\right)^{\frac{1}{p}}\\
&\leq&\left(\int_\Lambda|\nabla
u|^{p}\right)^{\frac{p-1}{p}}\cdot\left(\int_\Lambda|P_s|\nabla
u\|^p\d x\right)^{\frac{1}{p}}\\
&\leq& \int_\Lambda|\nabla u|^{p}\d x. \de

Using the same argument we obtain
$$ _{V^*}\langle -\eta_1 |u|^{\tilde{p}-2} u -\eta_2 u, T_n u
 \rangle_{V}\le \eta_2 \|u\|_n^2, \ u\in W^{1,p}(\Lambda).
$$
Hence (\ref{condition 1}) holds.

Note that  (\ref{noise condition}) also holds due to the same argument as in (\ref{e2}).
Therefore, the assertion follows from  Theorem \ref{T6.1}.

(2) If $\eta_2\le 0$, then by Lemma \ref{lem5} it is easy to show that $(H2^\prime)$ holds with $\beta=p$ (cf. \cite{L09,L10}).
Hence all assertions in Theorem \ref{thm:single_point}  hold for $(\ref{PLE})$.
 \end{proof}

\paragraph{Acknowledgement}
The authors would like to thank the referee for many helpful suggestions and comments.


\end{document}